\documentclass[reqno]{amsart}
\usepackage{amssymb, latexsym}
\usepackage{xspace}
\usepackage[bookmarksnumbered,colorlinks]{hyperref}
\usepackage{graphics}

\def\bb#1\eb{\textcolor{blue}
{#1}} %
\def\br#1\er{\textcolor{red}
{#1}} %
\def\bv#1\ev{\textcolor{green}
{#1}} %
\def\bc#1\ec{\textcolor{cyan}
{#1}} %

\usepackage{graphics}

\def\bal#1\eal{\begin{align}#1\end{align}}                      %
\def\baln#1\ealn{\begin{align*}#1\end{align*}}          
\def\bml#1\eml{\begin{multline}#1\end{multline}}        %
\def\bmln#1\emln{\begin{multline*}#1\end{multline*}}  %
\def\bga#1\ega{\begin{gather}#1\end{gather}}
\def\bgan#1\egan{\begin{gather*}#1\end{gather*}}

\makeatletter

\@addtoreset{equation}{section} \makeatother
\newtheorem{theorem}{Theorem}[section]
\newtheorem{lemma}[theorem]{Lemma}
\newtheorem{proposition}[theorem]{Proposition}

\newtheorem{definition}[theorem]{Definition}

\newtheorem{remark}[theorem]{Remark}

\newcommand{\be}{\begin{equation}}
\newcommand{\ee}{\end{equation}}
\newcommand{\h}{{\mathcal{H}}}
\newcommand{\M}{{\mathcal M}}
\newcommand{\J}{{\mathcal J}}
\newcommand{\elle}{{\mathcal L}}
\newcommand{\<}{\langle}
\renewcommand{\>}{\rangle}

\newcommand{\de}{\mathrm{d}}                        

\newcommand{\N}{\ensuremath{\mathbb{N}}\xspace}     
\newcommand{\R}{\ensuremath{\mathbb{R}}\xspace}     

\title[Connection by geodesics]{Connection by geodesics\\ on globally hyperbolic spacetimes \\
with a lightlike Killing vector field}

\author[R. Bartolo]{R. Bartolo$^1$}\thanks{$^{1,2}$ Partially supported by G.N.A.M.P.A. Research Project 2012
``Analisi Geometrica sulle variet\`a di Lorentz e applicazioni alla
Relativit\`a Generale''.}
\address{Rossella Bartolo\hfill\break\indent
Dipartimento di Meccanica, Matematica e Management\hfill\break\indent
Politecnico di Bari\hfill\break\indent
Via E. Orabona 4, 70125 Bari\hfill\break\indent
Italy}
\email{rossella.bartolo@poliba.it}

\author[A.M. Candela]{A.M. Candela$^2$}\thanks{$^{2,3}$ Partially supported by
the Spanish MICINN Grant FEDER funds MTM2010-18099, with FEDER
funds.}
\address{Anna Maria Candela \hfill\break\indent
Dipartimento di Matematica\hfill\break\indent
Universit\`a degli Studi di Bari ``Aldo Moro''\hfill\break\indent
Via  E. Orabona 4, 70125 Bari\hfill\break\indent
Italy}
\email{annamaria.candela@uniba.it}

\author[J.L. Flores]{J.L. Flores$^3$}\thanks{$^{3}$ Partially supported by
the Regional J. Andaluc\'{\i}a Grant P09-FQM-4496, with FEDER
funds.}
\address{Jos\'e Luis Flores \hfill\break\indent
 Departamento de \'Algebra, Geometr\'{\i}a y
Topolog\'{\i}a\hfill\break\indent Facultad de Ciencias,
Universidad de M\'alaga\hfill\break\indent Campus Teatinos, 29071
M\'alaga\hfill\break\indent Spain} \email{floresj@uma.es}

\subjclass[2000]{53C50, 53C22, 58E10}

\keywords{Lightlike vector field, global hyperbolicity, geodesic
connectedness, Killing vector field, Cauchy hypersurface, stationary spacetime,
gravitational wave, generalized plane wave.}

\date{}

\begin{document}

\begin{abstract}
Given a globally hyperbolic spacetime endowed
with a complete lightlike Killing vector field and a complete
Cauchy hypersurface, we characterize the points which can be
connected by geodesics. A straightforward consequence is the
geodesic connectedness of globally hyperbolic generalized plane
waves with a complete Cauchy hypersurface.
\end{abstract}

\maketitle

\section{Introduction}\label{s1}

During the past years there has been a considerable amount of
research related to the problem of geodesic connectedness of
Lorentzian manifolds (cf. the classical books \cite{bee,o}, the
updated survey \cite{cs} and references therein). This topic has
wide applications in Physics, but for mathematicians its interest
is essentially due to the peculiar difficulty of this natural
problem, which makes it challenging from both an analytical and a
geometrical point of view. In particular, a striking difference
with the Riemannian realm is that no analogous to the Hopf--Rinow
Theorem holds (for a counterexample, cf. \cite[Remark 1.14]{Pe} or
also \cite[p. 150 and Example 7.16]{o}). Thus, up to now,
sufficient conditions for geodesic connectedness have been
established only for a few models of Lorentzian spacetimes.

The ideas in the paper \cite{cfs} led to the following result (cf.
\cite[Theorem 1.1]{cfs}):

\begin{theorem}{\bf [Candela-Flores-S\'anchez]}\label{mainmain}
Let $({\mathcal L},\langle\cdot,\cdot\rangle_{L})$ be a stationary spacetime with a complete
timelike Killing vector field $K$. If ${\mathcal L}$ is globally
hyperbolic with a complete (smooth, spacelike) Cauchy hypersurface
$S$, then it is geodesically connected.
\end{theorem}

The interest of this theorem does not only rely on the intrinsic
geometric character and accuracy of its hypotheses (cf.
\cite[Section 6.3]{cfs}), but also on the fact that it is the top
result of a series of works on geodesic connectedness for standard
stationary spacetimes (cf. \cite{bcf, bf,gm,gp,pi}). If one
analyzes the extrinsic hypotheses under which standard stationary
spacetimes become globally hyperbolic (cf. \cite[Corollary
3.4]{s}) and the ones under which they become geodesically
connected (for instance, \cite[Theorem 1.2]{bcf}), one realizes
that the former imply the latter. So, it was natural to wonder if
global hyperbolicity implies geodesic connectedness for
stationary spacetimes, as Theorem \ref{mainmain} finally
confirmed.

Now, observe that Theorem \ref{mainmain} admits a natural limit
case, which consists of assuming the existence of a lightlike,
instead of timelike, Killing vector field. A remarkable family of
spacetimes which falls under this hypothesis is the class of
generalized plane waves. The geodesic connectedness and global
hyperbolicity of these spacetimes have been also studied. In this
case, one also finds that the extrinsic hypotheses which ensure
global hyperbolicity (see \cite[Theorem 4.1]{FS}) imply geodesic
connectedness (see \cite[Corollary 4.5]{cfs1}). So, a natural
question is if Theorem \ref{mainmain} still holds when the Killing
vector field $K$ is lightlike, instead of timelike; i.e.,

\smallskip

{\em taking any globally hyperbolic spacetime endowed with a
complete lightlike Killing vector field and a complete (smooth,
spacelike) Cauchy hypersurface, is it geodesically connected?}

\smallskip

\noindent In general, the answer to this question is negative
(cf. Section \ref{s8} (c)); however, we can characterize
which points can be connected by geodesics in this class of
spacetimes. More precisely, here we prove the following statement:

\begin{theorem} \label{tm}
Let $(\elle,\<\cdot,\cdot\>_L)$ be a globally hyperbolic spacetime
endowed with a complete lightlike Killing vector field $K$ and a
complete (smooth, spacelike) Cauchy hypersurface $S$. Given two
points $p, q \in \elle$, the following statements are equivalent:
\begin{itemize}
\item[{\sl (i)}] $p$ and $q$ are geodesically connected in $\elle$;
\item[{\sl (ii)}] $p$ and $q$ can be connected by a $C^1$ curve $\varphi$ on
$\elle$ such that $\langle \dot\varphi, K(\varphi)\rangle_L$ has
constant sign or  is identically equal to $0$.
\end{itemize}
\end{theorem}

Alike Theorem \ref{mainmain}, this result is intrinsic, sharp and
natural. Moreover, it presents nice consistency with previous
results on geodesic connectedness for generalized plane waves. The
proof is based on a limit argument. First, one perturbs the metric
of the spacetime into a sequence of standard stationary metrics
which approach to the original one. Given two points, one uses an
adapted version of Theorem \ref{mainmain} to ensure that they are
geodesically connected for sufficiently advanced metrics of the
sequence. Then, one uses property {\sl (ii)} to provide some
estimates on the sequence of connecting geodesics. Finally, a
thorough limit argument based on these estimates ensures the
existence of a limit connecting geodesic for the original metric.

Besides the geodesic connectedness, other geodesic properties of
stationary spacetimes have been studied in the last decades.
Theorem \ref{tm} inaugurates an interesting line of research
consisting of translating geodesic properties, from stationary
spacetimes to spacetimes with a lightlike Killing vector field, by
using a limit argument similar to the one developed below. The
fine estimates needed to overcome this procedure for the geodesic
connectedness problem, and the fact that this property is only
partially preserved when passing to the limit, suggest that, in
general, this line of research will be an interesting mathematical
challenge.

The rest of this paper is organized as follows. In Section
\ref{sadd} we recall some notations, definitions and background
tools on Lorentzian manifolds, especially on standard stationary
spacetimes. In Section \ref{ss} we explain the main arguments
involved in the intrinsic variational approach to the geodesic
connectedness problem in a stationary spacetime, when a global
splitting is not given a priori. The machinery developed in
Section \ref{ss} is used in Section \ref{s4} to prove Theorem
\ref{nuovo}, an adapted version of Theorem \ref{mainmain}. In
Section \ref{ausi} we apply Theorem \ref{nuovo} to a sequence of
standard stationary spacetimes obtained by perturbing the original
metric. As a consequence, fixed two arbitrary points, a sequence
of connecting geodesics of the perturbed metrics is obtained
(Proposition \ref{p1}). Then, in Section \ref{s3} we deduce some
estimates for these geodesics (Lemmas \ref{bounded} and
\ref{boundedt}) and apply a limit argument to them (Lemma
\ref{strong}) in order to prove Theorem \ref{tm}. The accuracy of
the hypotheses of Theorem \ref{tm} is showed in Section \ref{s8}.
Finally, in Section \ref{s7}, we provide some straightforward
applications of Theorem \ref{tm}, such as the Avez--Seifert result
in this ambient (Proposition \ref{pp}) and the geodesic
connectedness of some generalized plane waves (Theorem \ref{c0}).

\section{Notation and background tools}\label{sadd}

In this section we review some basic notions in Lo\-ren\-tzian
Geo\-me\-try used thro\-ughout the pa\-per (we re\-fer to
\cite{bee, o} for more details).

A {\em Lorentzian manifold} $({\mathcal L},\<\cdot,\cdot\>_L)$
(henceforth often simply denoted by $\mathcal L$) is a smooth
(connected) finite dimensional mani\-fold ${\mathcal L}$ equipped
with a symmetric non--degenerate tensor field $\<\cdot,\cdot\>_L$
of type $(0,2)$ with index $1$. A tangent vector $\zeta\in T_z
{\mathcal L}$ is called {\em timelike} (resp. {\em lightlike};
{\em spacelike}; {\em causal}) if $\<\zeta,\zeta\>_L<0$ (resp.
$\<\zeta,\zeta\>_L=0$ and $\zeta\neq 0$; $\<\zeta,\zeta\>_L>0$ or
$\zeta=0$; $\zeta$ is either timelike or lightlike). The set of
causal vectors at each tangent space has a structure of ``double
cone'' called {\em causal cones}.

A $C^1$ curve $\gamma:I\rightarrow {\mathcal L}$ ($I$ real
interval) is called {\em timelike} (resp. {\em lightlike}; {\em
spacelike}; {\em causal}) when so is $\dot{\gamma}(s)$ for all
$s\in I$. For causal curves, the definition is extended to include
piecewise $C^1$ curves: in this case, the two limit tangent
vectors on the breaks must belong to the same causal cone.

A smooth curve $\gamma:I\rightarrow {\mathcal L}$ is a
{\em geodesic} if it satisfies the equation
\[
D^L_{s}\dot \gamma\ =\ 0,
\]
where $D^L_s$ is the covariant derivative along $\gamma$
associated to the Levi--Civita connection of metric
$\<\cdot,\cdot\>_L$. Any geodesic $\gamma$
satisfies the conservation law
\[
\<\dot \gamma(s),\dot \gamma(s)\>_L \ \equiv\ E_{\gamma} \quad
\hbox{for some constant $E_{\gamma} \in \R$ and all $s \in I$.}
\]
So, its causal character can be directly re-written in terms of
the sign of $E_{\gamma}$. Two points $p$, $q \in \elle$ are {\em
geodesically connected} if there exists a geodesic $\gamma: I
\rightarrow\elle$ such that $\gamma(0) =p$ and $\gamma(1) = q$ (hereafter, $I:=[0,1]$).
This property is equivalent to a variational problem: namely, the
existence of a critical point of the {\em action functional}
\begin{equation}\label{fuu}
f(z)=\frac{1}{2}\int_0^1\langle\dot{z},\dot{z}\rangle_L \, \de s
\end{equation}
in the set $C^1(I,{\mathcal L})$ of all the $C^1$ curves $z: I
\rightarrow\elle$ such that $z(0) =p$ and $z(1) = q$.

A vector field $K$ in $\elle$ is
said {\em complete} if its integral curves are defined on the
whole real line. On the other hand, $K$ is said {\em Killing} if
one of the following equivalent statements holds (cf.
\cite[Propositions 9.23 and 9.25]{o}):
\begin{itemize}
\item[{\sl (i)}] the stages of its local flow consist of
isometries;
\item[{\sl (ii)}] the Lie derivative of
$\<\cdot,\cdot\>_L$ in the direction of $K$ is $0$;
\item[{\sl (iii)}] $\<D_X K,Y\>_L = - \<D_Y K,X\>_L$ for all vector fields
$X,Y$ on $\elle$.
\end{itemize}
If $K$ is a Killing vector field and $\gamma : I \to {\mathcal
L}$ is a geodesic, then there exists
$C_\gamma\in\R$ such that
\begin{equation}\label{vinc}
\<\dot \gamma(s),K(\gamma(s))\>_L\ \equiv\ C_\gamma \quad
\hbox{for all $s \in I$.}
\end{equation}

A {\em spacetime} is a Lorentzian manifold ${\mathcal L}$ with a
prescribed {\em time--orien\-ta\-tion}, that is, a continuous
choice of a causal cone at each point of $\mathcal L$, called {\em
future cone}, in opposition to the non--chosen one, named {\em
past cone}. A causal curve $\gamma$ in a spacetime is called {\em
future} or {\em past directed} depending on the time orientation
of the cone determined by $\dot\gamma$ at each point. Given
$p,q\in\elle$, we say that $p$ is in the causal past of $q$, and
we write $p<q$, if there exists a future--directed causal curve
from $p$ to $q$. Moreover, we denote by $p\leq q$ either $p<q$ or
$p=q$. For each $p\in\elle$, the {\em causal past} $J^-(p)$ and
the {\em causal future} $J^+(p)$ are defined as
\[
J^-(p)\ =\ \{q\in\elle:\ q\leq p\} \quad \hbox{and}\quad J^+(p)\
=\ \{q\in\elle:\ p\leq q\}.
\]

\begin{remark}{\em The causal relations allow one to extend the space
of piecewise $C^1$ causal curves to the space of (non--necessarily smooth) continuous
causal curves, in a way which is appropriate for convergence of curves. Actually, such curves have $H^1$ regularity
(cf. \cite[p. 54]{bee}, \cite[p. 442]{Ge} and also \cite[Definition 2.1, Remarks
2.2 and A.4]{cfs}).}
\end{remark}

A spacetime is called {\em stationary} if it admits a timelike
Killing vector field. There are several equivalent definitions of
global hyperbolicity for a spacetime (cf., e.g., \cite{ms}). Here,
we adopt the following: a spacetime is {\em globally hyperbolic}
if it contains a {\em Cauchy surface}, that is, a subset which is
crossed exactly once by any inextendible timelike curve. According
to the remarkable paper \cite{BeS1}, the Cauchy surface can be
chosen to be a smooth, spacelike hypersurface. In general, any
inextendible causal curve crosses (possibly, along a segment) a
Cauchy surface $S$; if, in addition, $S$ is spacelike (at least
$C^1$), then it crosses $S$ exactly once (cf. \cite[p. 342]{ms}).
Another important property of a spacetime $\elle$ admitting a
Cauchy surface $S$ is that $J^-(p)\cap S$ is compact for every
$p\in \elle$ (cf. \cite[Proposition 6.6.6]{HE}).

In this paper we are concerned with globally hyperbolic spacetimes
admitting a complete causal Killing vector field. The following
proposition, which slightly extends \cite[Theorem 2.3]{cfs},
provides a precise description of the structure of these
spacetimes.

\begin{proposition}\label{condelta}
Let $(\elle,\langle\cdot,\cdot\rangle_L)$ be a globally hyperbolic
spacetime admitting a complete causal Killing vector field $K$.
Then, there exist a Riemannian manifold
$(S,\langle\cdot,\cdot\rangle)$, a differentiable vector field
$\delta$ on $S$ and a differentiable non--negative function
$\beta$ on $S$ such that
\begin{equation}\label{metrica}
\elle=S\times\R\quad \hbox{and}\quad \<\zeta,\zeta'\>_L =
\<\xi,\xi'\> + \<\delta(x),\xi\>\tau' +
\<\delta(x),\xi'\>\tau-\beta(x)\tau\tau',
\end{equation}
for all $ z = (x,t) \in {\mathcal  L}$ and $\zeta = (\xi,\tau),
\zeta'
= (\xi',\tau')\in T_z{\mathcal L} = T_xS \times \R$.\\
Furthermore, if  $K$ is timelike then $\beta$ is non--vanishing,
i.e., $\beta(x) > 0$ for all $x \in S$; if $K$ is lightlike then
$\beta \equiv 0$, $\delta$ is non--vanishing and the metric on
$\elle$ becomes
\begin{equation}\label{metrica1}
\<\zeta,\zeta'\>_L = \<\xi,\xi'\> + \<\delta(x),\xi\>\tau' +
\<\delta(x),\xi'\>\tau,
\end{equation}
for all $ z = (x,t) \in {\mathcal  L}$ and $\zeta = (\xi,\tau),
\zeta' = (\xi',\tau')\in T_z{\mathcal L}= T_xS \times \R$.
\end{proposition}

\begin{proof}
Since $\elle$ is a globally
hyperbolic spacetime, it admits a spacelike Cauchy hypersurface
$S$ which becomes a Riemannian manifold when endowed with the induced metric
$\langle\cdot,\cdot\rangle$ from $\langle\cdot,\cdot\rangle_L$. Let us consider the map
\[
\Psi: (x,t)\in S\times \R\mapsto\Psi_t(x)\in \elle,
\]
being $\Psi$ the flow of the complete vector field $K$. Since $K$
is causal, its integral curves are also causal. So, each point of
$\elle$ is crossed by one integral curve of $K$, which crosses $S$
at exactly one point. Therefore, $\Psi$ is a diffeomorphism. As
$K$ is Killing, the pull--back metric
$\Psi^{*}\langle\cdot,\cdot\rangle_{L}$ is independent of $t$.
Hence, taking $\beta(x)=-\langle K(z),K(z)\rangle_L$ and denoting
by $\delta(x)$ the orthogonal projection of $K(z)$ on $T_xS$ for
any $z=(x,t)\in S \times \{t\}$, the metric expression
\eqref{metrica} follows.\\
Furthermore, if $K$ is timelike, then $\beta$ is clearly strictly
positive; instead, if $K$ is lightlike, then $\beta \equiv 0$ and
$\delta$ is non--vanishing (since $K(z)$ cannot be orthogonal to
$T_x S$).
\end{proof}

\begin{remark}\label{remark}{\em For further use, here we
emphasize the following relations, contained in the proof of
previous proposition: for any $z=(x,t)\in S\times \R$ we have
\[
\begin{array}{l}
K\equiv \partial_t,\quad S\equiv S\times\{0\}, \quad \beta(x)=-\langle K(z),K(z)\rangle_{L},\\
\delta(x)\equiv\; \text{orthogonal projection of $K(z)$ on $T_xS$.}
\end{array}
\]
}
\end{remark}

In general, a spacetime as in (\ref{metrica}) with $\beta(x) > 0$
on $S$ is called {\em standard stationary}. For this class of
spacetimes, $K=\partial_t$ is always a complete timelike Killing
vector field. A smooth curve $\gamma=(x,t)$ in a standard
stationary spacetime $\elle$ is a geodesic if and only if it
satisfies the following system of differential equations:
\begin{equation}\label{equa}
\left\{ \begin{array}{ll} { \displaystyle D_s\dot x - \dot
t\,F(x)[\dot x] + \ddot t\,\delta(x)+
{\frac 12}{\dot t}^2\nabla\beta(x) = 0} \\
{\displaystyle \frac{\de}{\de s}\left(\beta(x)\dot t -
\<\delta(x),\dot x\>\right) = 0,}
      \end{array}
      \right.
      \end{equation}
where $D_s$ denotes the covariant derivative along $x$ associated
to the Levi--Civita connection of metric
$\langle\cdot,\cdot\rangle$, and $F(x)$ denotes the linear
(continuous) operator on $T_x S$ associated to the bilinear form
\[
{\rm curl} \,\delta(x)[\xi,\xi'] = \<(\delta'(x))^T[\xi],\xi'\> -
\<\delta'(x)[\xi'],\xi\> \qquad \hbox{for all } \xi,\xi'\in T_x S,
\]
being $\delta'(x)$ the differential map of $\delta(x)$ and
$(\delta'(x))^T$ its transpose (cf., e.g., \cite[Appendix A]{[BGS3]}).

We conclude this section with the following result, which will be
used later on in the paper:

\begin{proposition}\label{p} Let $({\mathcal L},\<\cdot,\cdot\>_{L})$ be a standard stationary
spacetime as in (\ref{metrica}) and
$(S,\langle\cdot,\cdot\rangle)$ a complete Riemannian manifold. Given two points
$p=(x_p,t_p)$, $q=(x_q,t_q)\in\elle$ satisfying $\Delta_t=
t_q-t_p\geq 0$, the following assertions hold:
\begin{itemize}
\item[{\sl (i)}] $J^-(q)\cap (S\times \{t_p\})$ is closed in
$S\times \{t_p\}$;
\item[{\sl (ii)}] if $J^-(q)\cap (S\times
\{t_p\})$ is compact in $S\times \{t_p\}$, then there exists
$\varepsilon>0$ such that, setting
$q_{\varepsilon}=(x_q,t_q+\varepsilon)$,
$J^-(q_{\varepsilon})\cap
(S\times \{t_p\})$  is
also compact in $S\times \{t_p\}$.
\end{itemize}
\end{proposition}

\begin{proof} {\sl (i)}
Arguing by contradiction, assume that $J^-(q)\cap (S\times \{t_p\})$ is not closed in $S\times
\{t_p\}$. Then, there exists a sequence $(y_k)_k \subset
J^-(q)\cap (S\times \{t_p\})$ converging to some point $y \in
S\times \{t_p\}$, but \be\label{uff1} y \not\in J^-(q). \ee By
assumption, for each $k \in \N$ there exists a past
inextendible\footnote{The past inextendible causal curves
$\gamma_k$ can be obtained by prolonging the corresponding causal
curves from $q$ to $y_k$ (ensured by condition $y_k\in J^-(q)$)
with integral lines of the timelike vector field $-\partial_t$.}
causal curve $\gamma_k$ departing from $q$ and passing through
$y_k$. Then, \cite[Proposition 3.31]{bee} ensures that, up to a
subsequence, $(\gamma_k)_k$ converges to a past inextendible
causal curve $\gamma$ departing from $q$ and passing
through $y$. Therefore, $y\in J^-(q)$, in contradiction with \eqref{uff1}.\\
{\sl (ii)} By contradiction, assume the existence of a sequence of
points $(q_n)_n$, with $q_n=(x_q,t_q+\varepsilon_n)\in\elle$ and
$\varepsilon_n\searrow 0$, such that for all $n \in \N$ the set
$J^-(q_{n})\cap (S\times \{t_p\})$ is not compact in $S\times
\{t_p\}$. By the Hopf--Rinow theorem, since
$(S,\langle\cdot,\cdot\rangle)$ is complete and $J^-(q_{n})\cap
(S\times \{t_p\})$ is closed (property {\sl (i)}), it cannot be
bounded. So, for every $n \in \N$ there exists an unbounded
sequence of points $(p^n_k)_k\subset J^-(q_{n})\cap (S\times
\{t_p\})$, with $p^n_k=(x^n_k,t_p)$. By using a Cantor's diagonal
type argument, we construct an unbounded sequence $(p_n)_n$, with
$p_{n}=p^n_{k_n}$, such that $p_n\in J^-(q_{n})\cap (S\times
\{t_p\})$ for all $n$. Denote by $\gamma_n=(x_n,t_n)$ a
future--directed causal curve joining $p_n$ to $q_n$, and let
$s_n\in I$ be such that $t_n(s_n)=t_p+\varepsilon_n$ for each $n
\in\N$. Since the future--directed causal curve
$\alpha_n=(x_n,t_n-\varepsilon_n)$ on $[s_n,1]$ joins
$z_n=(x_{n}(s_n),t_p)$ to $q$, we have that $(z_n)_n$ is
contained in the compact set $J^-(q)\cap (S\times \{t_p\})$.
Thus, since $(p_n)_n$ is unbounded in $S\times \{t_p\}$, there exists
$\overline{s}_n\in [0,s_n]$ such that
\begin{equation}\label{eeq}
x_n\mid_{[\overline{s}_n,s_n]}\;\;\hbox{remains
bounded}\quad\hbox{and}\quad {\rm
length}(x_n\mid_{[\overline{s}_n,s_n]})\geq 1\;\; \forall n\in\N.
\end{equation}
On the other hand, as $\gamma_n=(x_n,t_n)$ is causal and
future--directed, $t_n$ is characterized by
$\<\dot\gamma_n,\dot\gamma_n\>_L\leq 0$ and $\dot t_n>0$ on $I$
(recall \eqref{metrica}), hence it follows that
\[
\dot{t}_n\geq
\frac{\langle\delta(x_n),\dot{x}_n\rangle}{\beta(x_n)} +
\sqrt{\frac{\langle\delta(x_n),\dot{x}_n\rangle^2}{\beta(x_n)^2} +
\frac{\langle\dot{x}_n,\dot{x}_n\rangle}{\beta(x_n)}}\quad\hbox{on
$I$.}
\]
By integrating the previous inequality in $[\bar{s}_n,s_n]$ , we deduce
\[
\int_{\overline{s}_n}^{s_n}\frac{\langle\delta(x_n),\dot{x}_n\rangle}{\beta(x_n)}\, \de s
+
\int_{\overline{s}_n}^{s_n}\sqrt{\frac{\langle\delta(x_n),\dot{x}_n\rangle^2}{\beta(x_n)^2}
+ \frac{\langle\dot{x}_n,\dot{x}_n\rangle}{\beta(x_n)}}\, \de s\leq
\int_{\overline{s}_n}^{s_n}\dot{t}_n\, \de s\leq\varepsilon_n\rightarrow
0,
\]
as $n\rightarrow\infty$. However, by virtue of (\ref{eeq}), the
first member of the previous expression remains positive and far from
zero, a contradiction.
\end{proof}


\section{Stationary intrinsic functional framework}\label{ss}

A considerable contribution to the study of the geodesic
connectedness of spacetimes was given in \cite{gm}. In that paper
the authors introduced a variational principle for geodesics,
based on the natural constraint \eqref{vinc}, and proved the
geodesic connectedness of standard stationary spacetimes $\elle$,
under some boundedness assumptions for the metric coefficients
$|\delta|$ and $\beta$ (recall \eqref{metrica}). Under the
hypotheses of Theorem \ref{mainmain}, the spacetime $\mathcal L$
globally splits into (\ref{metrica}), and previous result can be
applied. However, this splitting is neither unique nor canonically
associated to $\elle$, and the conclusion may depend on it. In
order to avoid this arbitrariness, an intrinsic approach to the
problem of geodes\-ic connectedness was developed in \cite{gp}.
There, the variational principle in \cite{gm} is translated into a
splitting independent form, and a compactness assumption on the
infinite dimensional manifold of the paths between two points is
introduced, called {\em pseudo--coercivity} (see from Theorem
\ref{t1} till the end of this section). This condition implies
global hyperbolicity, but, in the practice, it is quite difficult
to verify. Motivated by this deficiency, in \cite{cfs} the authors
worked under intrinsic geometric assumptions, which involve the
causal structure of the spacetime and are shown to be equivalent
to pseudo--coercivity. For a given complete spacelike smooth
Cauchy hypersurface $S$ and a given complete timelike Killing
vector field $K$, Proposition \ref{condelta} is applied to obtain
the corresponding global splitting. But, even if this splitting is
neither unique nor canonically associated to $\elle$, the result
obtained in \cite{cfs} is independent of the chosen $K$ and $S$,
and no growth hypotheses on the coefficients of the metric
$\<\cdot,\cdot\>_{L}$ are involved.

As we will see later on, the proof of Theorem \ref{tm} makes use
of Theorem \ref{nuovo}, a refinement of Theorem \ref{mainmain}.
So, in the rest of this section we are going to recall the
intrinsic variational functional framework associated to a
stationary spacetime, as developed in \cite{cfs, gp}.

Let $(\elle,\langle\cdot,\cdot\rangle_{L})$ be a
stationary spacetime. As shown in \cite{gp},
by taking into account the constraint (\ref{vinc}), the geodesics
in $\elle$ connecting two fixed
points $p,q\in\elle$ correspond to critical points of functional
$f$ in \eqref{fuu} restricted to the set of curves
\[
C^1_{K}(p,q) = \{z\in C^1(I,\elle) :  \exists\, C_z \in \R\
\hbox{such that $\langle\dot z,K(z)\rangle_{L} \equiv C_z$}\}.
\]
Since our approach will require dealing with $H^1$ curves on
$\elle$, we also introduce the infinite dimensional manifold
\[
\begin{split}
\Omega(p,q)\ =\ \big\{z:I\rightarrow \elle :\;\, & z\ \text{ is absolutely continuous,}\\
& z(0) =p,\, z(1) =q,\, \int_0^1\<\dot z, \dot z\>_{R}\, \de s
<+\infty\big\},
\end{split}
\]
where $\langle\cdot,\cdot\rangle_{R}$ is the Riemannian metric
canonically associated to $K$ and $\langle\cdot,\cdot\rangle_{L}$,
i.e.
\[ \<\zeta, \zeta'\>_{R} =\ \<\zeta, \zeta'\>_{L} -\ 2\
\frac{\<\zeta, K(z)\>_{L}\ \< \zeta', K(z)\>_{L}}{ \<K(z) , K(z)
\>_{L}}\qquad \hbox{ for all } z \in \elle,\;\; \zeta, \zeta' \in T_z
\elle.
\]
For each $z\in\Omega(p,q)$ the tangent space $T_z\Omega(p,q)$ is
given by the $H^1$ vector fields $\zeta\colon I\to T\mathcal
\elle$ along $z$ such that $\zeta(0)=0=\zeta(1)$. Moreover, the functional
$f$ in (\ref{fuu}) is well defined and finite on the whole manifold
$\Omega(p,q)$. Standard arguments ensure that $f$ is smooth, with
differential given by
\[
\de f(z)[\zeta] = \int_0^1 \langle\dot z,\nabla_s^L\zeta\rangle_L \de
s \qquad \hbox{ for all }  z\in \Omega(p,q),\ \zeta \in T_z\Omega(p,q),
\]
and its critical points are the geodesics in
$(\elle,\langle\cdot,\cdot\rangle_{L})$ connecting $p$ to $q$.

The set $C^1_K(p,q)$ can be also extended to a subset of
$\Omega(p,q)$ defined as \be\label{omega} \Omega_{K}(p,q) =
\left\{ z \in \Omega (p,q) :\ \exists\, C_z\in \R \ \hbox{such
that}\ \<\dot z, K(z)\>_{L}\equiv C_z\ \text{a.e. on $I$}\right\}
\ee
and definitions and theorems below hold on both of them.

The following result reduces the geodesic connectedness problem
between $p$ and $q$ to the search of critical points of $f$ on
$\Omega_K(p,q)$ (cf. \cite[Theorem 3.3]{gp}):
\begin{theorem} \label{t1}
A curve $\gamma\in\Omega(p,q)$  is a geodesic on $\elle$ connecting $p$ to $q$ if and only if
$\gamma\in\Omega_K(p,q)$ and $\gamma$ is a critical point of $f$ in (\ref{fuu}) restricted to
$\Omega_{K}(p,q)$.
\end{theorem}

The following definitions are given in \cite{gp}:
\begin{itemize}
\item[{\sl (i)}] given $c\in \R$, the set $\Omega_{K}(p,q)$ is
{\em $c$--precompact} for $f$ if every sequence $(z_m)_m$ in
$\Omega_{K}(p,q)$ such that $f(z_m) \le c$ has a subsequence
which converges weakly in $\Omega_K(p,q)$ (hence, uniformly in
$\elle$);
\item[{\sl (ii)}] the restriction of $f$ to
$\Omega_K(p,q)$ is {\em pseudo--coercive} if $\Omega_K(p,q)$ is
$c$--pre\-com\-pact for all $c \ge \inf f(\Omega_K(p,q))$.
\end{itemize}
Then, the following theorem holds (cf. \cite[Theorem 1.2]{gp}).

\begin{theorem}\label{intrinsictheo}{\bf [Giannoni-Piccione]}
If $\Omega_K(p,q)$ is not empty and there exists $c > \inf
f(\Omega_K(p,q))$ such that $\Omega_{K}(p,q)$ is $c$--precompact, then
there exists at least one geodesic in $({\mathcal
L},\<\cdot,\cdot\>_L)$ joining $p$ to $q$.
\end{theorem}

\begin{remark} \label{noempty}
{\em In the hypotheses of Theorem \ref{mainmain},
the completeness of $K$ guarantees that $\Omega_K(p,q) \ne
\emptyset$ for any $p$, $q \in \elle$ (cf. \cite[Lemma 5.7]{gp}
and \cite[Proposition 3.6]{cfs}); moreover, the
technical condition of pseudo--coercivity holds (cf. \cite[Theorem
5.1]{cfs}). Therefore, Theorem \ref{mainmain} follows from Theorem
\ref{intrinsictheo}. }
\end{remark}

\section{The stationary non-canonical global splitting}\label{s4}

Let $({\mathcal L},\<\cdot,\cdot\>_L)$ be a standard stationary
spacetime as in (\ref{metrica}) with $\beta(x) > 0$ for all $x \in S$. Given two points $p=(x_p,t_p),
q=(x_q,t_q)\in\elle$, the space $\Omega(p,q)$ can be re-written as
\[
\Omega(p,q)\ =\ \Omega (x_p,x_q;S)\times W(t_p,t_q),
\]
where
\[
\begin{split}
\Omega (x_p,x_q;S)\ =\ \big\{x:I\rightarrow S:\, &\, x \text{ is absolutely continuous,}\\
&\, x(0) =x_p,\, x(1) =x_q,\, \int_0^1\<\dot x, \dot x\>\, \de s <+\infty\big\},
\end{split}
\]
\[
W(t_p,t_q)\ =\ \big\{t\in H^1(I,\R):\, t(0) =t_p,\, t(1) =t_q\big\}
= H^1_0(I,\R) + T^\ast,
\]
being $H^1(I,\R)$ the classical Sobolev space,
\[
H^1_0(I,\R)= \big\{t\in H^1(I,\R):\, t(0)=0=t(1)\big\}
\]
and \be\label{tempo} T^\ast: s\in I\longmapsto t_p +
s\Delta_t\in\R, \qquad \Delta_t= t_q-t_p. \ee For every
$x\in\Omega (x_p,x_q;S)$ it results
\[
\begin{split}
T_x\Omega (x_p,x_q;S)\ =\ \big\{\xi :I\rightarrow T_xS: \, &\, \xi\ \text{ is absolutely continuous,}\\
&\, \xi(0) = 0 =\xi(1), \, \int_0^1\<D_s\xi, D_s\xi\>\, \de s
<+\infty\big\}.
\end{split}
\]
Furthermore, $W(t_p,t_q)$ is a closed affine submanifold of
$H^1(I,\R)$ having tangent space
\[
T_t  W(t_p,t_q) = H^1_0(I,\R)\qquad \hbox{for all $t\in
W(t_p,t_q)$.}
\]
So, for every $z=(x,t)\in \Omega(p,q)$ it is
\[
T_z\Omega(p,q)= T_x\Omega (x_p,x_q;S) \times T_t  W(t_p,t_q) =
T_x\Omega (x_p,x_q;S) \times H^1_0(I,\R)
\]
and $\Omega(p,q)$ can be equipped with the Riemannian structure
\[
\langle\zeta,\zeta\rangle_H\ =\
\langle(\xi,\tau),(\xi,\tau)\rangle_H\ =\
\int_0^1\langle D_s\xi,D_s\xi\rangle\, \de s + \int_0^1\dot\tau^2\, \de s,
\]
for all $z=(x,t)\in\Omega(p,q)$ and $\zeta=(\xi,\tau)\in
T_z\Omega(p,q)$.

Next, assume that $(S,\<\cdot,\cdot\>)$ is complete. Then,
$\Omega(x_p,x_q;S)$ is a complete infinite dimensional manifold
(cf. \cite{k}). By Nash Embedding Theorem the complete mani\-fold
$S$ can be seen as a closed submanifold of an Euclidean space
$\R^N$ (cf. \cite{muller} for the existence of a closed isometric
embedding). Hence, $\Omega(x_p,x_q;S)$ is an embedded submanifold
of the classical Sobolev space $H^1(I,\R^N)$. As usual, let us set
\[
\|y\|^2 = \|y\|_2^2 + \|\dot y\|_2^2 \quad \hbox{for all } y\in H^1 (I, \R^N),
\]
where $\|\cdot\|_2$ denotes the standard $L^2$--norm. It is well
known that the following inequalities hold: \be\label{pi} \|y\|_2\
\leq\ \|y\|_\infty\ \le\ \|\dot y\|_2 \quad \hbox{for all $y \in
H_0^1 (I, \R^N)$,} \ee where $\|\cdot\|_\infty$ denotes the norm
of the uniform convergence (cf., e.g., \cite[Proposition 8.13]{br}). Moreover, the Ascoli--Arzel\'a Theorem
implies that any bounded sequence in $H^1 (I, \R^N)$ has a
uniformly converging subsequence in $C(I,\R^N)$.

For any absolutely continuous curve $z=(x,t):I\rightarrow
\elle$, one has
 \be\label{dopo}
 \<\dot z,K(z)\>_L = \<\dot z,\partial_t\>_L = \<\delta(x),\dot x\> - \beta(x)\dot
 t,\qquad\hbox{(recall that $K=\partial_t$).}
 \ee
Taking into account \eqref{dopo}, if $z\in\Omega_{K}(p,q)$ (recall
\eqref{omega}) then there exists a constant $C_z$ such that
\begin{equation}\label{dott}
\dot t =\ \frac{\langle\delta(x),\dot x\rangle -
C_z}{\beta(x)}\qquad\hbox{a.e. on $I$.}
\end{equation}
Thus, integrating both hand sides of (\ref{dott}) on $I$, and
isolating $C_z$, we get
\begin{equation} \label{cz}
C_z =\ \left(\int_0^1\frac{\langle\delta(x),\dot
x\rangle}{\beta(x)}\ \de s\ -\ \Delta_t\right)\
\left(\int_0^1\frac{\de s}{\beta(x)}\right)^{-1}.
\end{equation}
Denoting by $\J$ the restriction to $\Omega_{K}(p,q)$ of the
functional $f$ in (\ref{fuu}) with metric \eqref{metrica}, and
substituting (\ref{cz}) in (\ref{dott}), $\J$ can expressed as a
functional depending only on $\Delta_t$ (cf. \eqref{tempo}) and
the component $x$ of the curve $z=(x,t)\in\Omega_{K}(p,q)$:
\be\label{Jn'}
\begin{split}
\J(x)\ =\ &\frac 12\|\dot{x}\|_2^2 \\
&+\frac
12\left[\int_0^1\frac{\<\delta(x),\dot x\>^2}{\beta(x)}\de s \,
-\left(\int_0^1 \frac{\<\delta(x),\dot x\>}{\beta(x)}\de s\right)^2
\left(\int_0^1\frac{1}{\beta (x)}\de s\right)^{-1}\right]\\
&-\ \frac{\Delta_t}{2}
\left(\Delta_t - 2 \int_0^1\frac{\<\delta(x),\dot x\>}{\beta(x)}\de s \right)\
\left(\int_0^1\frac{1}{\beta(x)}\de s\right)^{-1}.
\end{split}
\ee
By construction, $f(z) = \J(x)$ if $z = (x,t) \in
\Omega_{K}(p,q)$; furthermore, by applying the Cauchy--Schwarz
inequality to the middle term of \eqref{Jn'}, we get
\be\label{Jn1}
2 \J(x)\ \ge\
\|\dot{x}\|_2^2\ -\ \Delta_t \left(\Delta_t - 2
\int_0^1\frac{\<\delta(x),\dot x\>}{\beta(x)}\de s \right)\
\left(\int_0^1\frac{1}{\beta(x)}\de s\right)^{-1}.
\ee

Now, we are ready to establish an adapted version of Theorem
\ref{mainmain}, needed in Section \ref{ausi}. But, first, we
recall the following result (cf. \cite[Lemma 5.4]{cfs}):
\begin{lemma}\label{l} Fixed any $x\in\Omega(x_p,x_q;S)\cap C^1(I,S)$ ($x$ non--constant if $x_p=x_q$)
there exists a unique future directed lightlike curve $\gamma^l
=(x^l,t^l): [0,1] \to \elle$ joining $(x_p,t_p)$ to $\{x_q\}\times
\R$ in a time $ T(x)= t^l(1)-t^l(0)>0$ such that $x^l=x$.
Moreover, $T(x)$ satisfies:
\begin{equation}
\label{time} T(x) = \int_0^1\frac{\langle\delta(x),\dot
x\rangle}{\beta(x)}\ \de s\ +
\int_0^1\frac{\sqrt{\langle\delta(x),\dot x\rangle^2 + \langle\dot
x,\dot x\rangle\beta(x)}}{\beta(x)}\ \de s.
\end{equation}
\end{lemma}

\begin{theorem}\label{nuovo}
Let $({\mathcal L},\<\cdot,\cdot\>_{L})$ be a standard stationary
spacetime as in (\ref{metrica}) and
$(S,\langle\cdot,\cdot\rangle)$ a complete Riemannian manifold. If two points
$p=(x_p,t_p)$, $q=(x_q,t_q)\in\elle$ satisfy
\[
\Delta_t= t_q-t_p\geq 0 \quad\hbox{and}\quad J^-(q)\cap (S\times
\{t_p\})\; \hbox{is compact,}
\]
then they are connected by a geodesic in ${\mathcal L}$.
\end{theorem}

\begin{proof}\footnote{Even if the core of this proof is essentially contained in \cite[Section
5]{cfs}, here we rearrange it for reader's convenience. Although
the functional $f$ is defined in $C^1_K(p,q)$, it is natural to
consider limits in $\Omega_{K}(p,q)$ (cf. \cite[p. 522 and Remark
3.3]{cfs}).} From Theorem \ref{intrinsictheo} and Remark
\ref{noempty}, it suffices to show that $f$ restricted to
$C^1_K(p,q)$ is $c$--precompact for some $c > \inf f(C^1_K(p,q))$,
i.e. every sequence $(z_m)_m$ in $C^1_K(p,q)$ such that
$(f(z_m))_m$ is upper bounded, has a uniformly convergent
subsequence. So, let us consider any $c > \inf f(C^1_K(p,q))$ and
a sequence of curves $(z_m)_m$ in $C^1_K(p,q)$, with
$z_m=(x_m,t_m)$, satisfying \be\label{hyp1} (f(z_m))_m \;
(\hbox{and thus}\; (\J(x_m))_m)\; \hbox{is upper bounded by $c$.}
\ee Setting
\[
C^1(x_p,x_q)\ =\ \Omega(x_p,x_q;S) \cap C^1(I, S),
\]
we have that
\[
(x_m)_m\subset C^1(x_p,x_q).
\]
It suffices to prove that
\begin{equation}\label{c1}
\hbox{$(\|\dot{x}_m\|_2)_m$ is bounded, up to a subsequence};
\end{equation}
indeed, by \eqref{pi} it follows that $(x_m)_m$ is bounded in $\Omega(x_p,x_q;S)$
and the supports of these curves are contained in a compact subset of $S$.
Hence, the Ascoli--Arzel\'a Theorem applies.\\
As we will see later, \eqref{c1} will be a direct consequence of the
following three claims.
\smallskip

\noindent {\em Claim 1.} If (\ref{c1}) does not hold, i.e.,
\begin{equation}\label{hyp2}
\|\dot x_m\|_2\ \to\ +\infty,
\end{equation}
then no compact subset of $S$ contains all the elements of the
sequence $(x_m)_m$.\\
{\em Proof of Claim 1.} Otherwise,  being
$(\beta(x_m))_m$ and $(|\delta(x_m)|)_m$ bounded
(with $|\delta(x_m)|^2=\left<\delta(x_m),\delta(x_m)\right>)$, by
\eqref{Jn1} and the Cauchy--Schwarz inequality it follows
$$
2\J(x_m)\geq \|\dot x_m\|^2_2- C_1\|\dot x_m\|_2 - C_2
$$
for some $C_1,C_2>0$ independent of $m\in\N$.
Hence  (\ref{hyp2}) implies
\begin{equation}\label{limit2}
\J(x_m) \to +\infty,
\end{equation}
in contradiction with
\eqref{hyp1}.
\smallskip

\noindent {\em Claim 2.} If no compact subset of $S$ contains all
the elements of the sequence $(x_m)_m$, then there exists some
$\varepsilon>0$ such that (recall (\ref{time}))
\be\label{hyp5} T_m:=T(x_m) > \Delta_t+\varepsilon\quad\hbox{for
infinitely many $m\in\N$.} \ee
{\em Proof of Claim 2.} Taking $\varepsilon>0$ provided
by Proposition \ref{p} {\sl (ii)}, let us assume by contradiction that
statement \eqref{hyp5} does not hold. This means that
\begin{equation}\label{er}
T_m\leq\Delta_t+\varepsilon\quad\hbox{for all $m$ big enough.}
\end{equation}
From Lemma \ref{l}, there exist future directed lightlike curves
$\gamma_m^l = (x_m,t^l_m)$ joining $p$ to $(x_q,t_p+T_m)$. Then,
from (\ref{er}), these curves can be prolonged with the integral
curves of $\partial_t$ to get future directed causal curves from
$p$ to $q_{\varepsilon}=(x_q,t_p+\Delta_t+\varepsilon)=(x_q,t_q+\varepsilon)$.
These  curves have support in $J^-(q_\varepsilon)$, so the curves
$(x_m,t_p)$ lie in the compact set $J^-(q_{\varepsilon})\cap
(S\times \{t_p\})$ (recall Proposition \ref{p} {\sl (ii)}), in
contradiction with the hypothesis.
\smallskip

\noindent{\em Claim 3.} Conditions \eqref{hyp2} and \eqref{hyp5}
imply \eqref{limit2}, up to a subsequence.\\
{\em Proof of Claim 3.}
If there exists a constant $c_1 > 0$ such that
\[
\left(\Delta_t - 2 \int_0^1\frac{\langle\delta(x_m),\dot
x_m\rangle}{\beta(x_m)}\ \de s\right) \left( \int_0^1\frac{\de
s}{\beta(x_m)}\right)^{-1} \le c_1 \quad \hbox{for infinitely many
$m\in\N$,}
\]
then the desired limit (\ref{limit2}) follows from \eqref{Jn1} and
\eqref{hyp2}.\\
Otherwise, assume that
\begin{equation}\label{limit3}
\left(\Delta_t - 2 \int_0^1\frac{\langle\delta(x_m),\dot
x_m\rangle}{\beta(x_m)}\ \de s\right) \left(
\int_0^1\frac{\de s}{\beta(x_m)}\right)^{-1} \ \longrightarrow\
+\infty\quad \hbox{as $m \to +\infty$.}
\end{equation}
Setting
\begin{eqnarray*}
\tilde T_m &=& \int_0^1\frac{\langle\delta(x_m),\dot x_m\rangle}{\beta(x_m)}\ \de s\\
&&+ \sqrt{\left( \int_0^1\frac{\langle\delta(x_m),\dot
x_m\rangle^2}{\beta(x_m)}\ \de s + \|\dot x_m\|^2\right)
\int_0^1\frac{\de s}{\beta(x_m)}},\
\end{eqnarray*}
the Cauchy--Schwarz inequality implies
\begin{equation}\label{5.4b}
T_m \le \tilde T_m \quad \forall m \in \N.
\end{equation}
Moreover,
\[
\begin{split}
&\int_0^1\frac{\langle\delta(x_m),\dot x_m\rangle^2}
{\beta(x_m)}\ \de s +
\|\dot x_m\|_2^2\\
&\qquad \displaystyle =\ \left(\tilde T_m -
\int_0^1\frac{\langle\delta(x_m),\dot x_m\rangle}{\beta(x_m)}\
\de s\right)^2 \left(\int_0^1\frac{\de s}{\beta(x_m)}\right)^{-1}.
\end{split}
\]
For infinitely many $m\in \N$, inequality \eqref{hyp5} holds and
\begin{equation}\label{ee}
\Delta_t - 2 \int_0^1\frac{\langle\delta(x_m),\dot
x_m\rangle}{\beta(x_m)}\ \de s\quad \hbox{is positive (recall
\eqref{limit3}).}
\end{equation}
Hence,
\begin{eqnarray*}
2\J(x_m) &=& \left(\tilde T_m -
\int_0^1\frac{\langle\delta(x_m),\dot x_m\rangle}{\beta(x_m)}\ \de
s\right)^2
\left(\int_0^1\frac{\de s}{\beta(x_m)}\right)^{-1}\\
&&-\ \left(\int_0^1\frac{\langle\delta(x_m),\dot
x_m\rangle}{\beta(x_m)}\ \de s\ -\ \Delta_t\right)^2
\left(\int_0^1\frac{\de s}{\beta(x_m)}\right)^{-1}\\
&=& \left(\tilde T_m^2 - \Delta_t^2 - 2 (\tilde T_m - \Delta_t)
\int_0^1\frac{\langle\delta(x_m),\dot x_m\rangle}{\beta(x_m)}\
\de s\right)\
\left(\int_0^1\frac{\de s}{\beta(x_m)}\right)^{-1}\\
&=& (\tilde T_m - \Delta_t)\ \left(\tilde T_m + \Delta_t - 2
\int_0^1\frac{\langle\delta(x_m),\dot x_m\rangle}{\beta(x_m)}\ \de
s\right)\ \left(\int_0^1\frac{\de s}{\beta(x_m)}\right)^{-1}\\
&\geq &
\varepsilon\left[ \tilde{T}_m
 +  \left(\Delta_t - 2
\int_0^1\frac{\langle\delta(x_m),\dot x_m\rangle}{\beta(x_m)}\ \de
s\right)\right] \left(\int_0^1\frac{\de s}{\beta(x_m)}\right)^{-1}\\
&\geq & \varepsilon \left(\Delta_t - 2
\int_0^1\frac{\langle\delta(x_m),\dot x_m\rangle}{\beta(x_m)}\ \de
s\right)\ \left(\int_0^1\frac{\de s}{\beta(x_m)}\right)^{-1},
\end{eqnarray*}
where, in the first inequality, we have taken into account
\eqref{hyp5}, (\ref{5.4b}) and (\ref{ee}). So, the limit
(\ref{limit3}) clearly implies the limit (\ref{limit2}), up to a subsequence.
\smallskip

\noindent
Summing up, if (\ref{c1}) does not hold, Claim 1 ensures that no
compact subset of $S$ contains all the elements of the sequence
$(x_m)_m$. Then, Claims 2 and 3 imply (\ref{limit2}), up to a
subsequence, in contradiction with \eqref{hyp1}.
\end{proof}

\section{Connecting geodesics in auxiliary stationary spacetimes}\label{ausi}

Throughout this section, $(\elle,\langle\cdot,\cdot\rangle_L)$
will be a spacetime which satisfies the hypotheses of Theorem \ref{tm}. From Proposition
\ref{condelta}, $\elle = S \times \R$ and
$\langle\cdot,\cdot\rangle_L$ is as in (\ref{metrica1}), with
metric coefficients given by Remark \ref{remark}.

For each $n\in\N$, let us consider the standard stationary
spacetime $(\elle_n,\langle\cdot,\cdot\rangle_n)$ (often simply
denoted by $\elle_n$), where $\elle_n = \elle$ and
\be\label{stati}
\langle\zeta,\zeta'\rangle_{n}\ =\
\langle\zeta,\zeta'\rangle_L -\,\frac{1}{n}\tau\tau'  =\
\langle\xi,\xi'\rangle+\langle\delta(x),\xi\rangle \, \tau' +
\langle\delta(x),\xi'\rangle \, \tau-\,\frac{1}{n}\tau\tau'
\ee
for any $ z = (x,t) \in {\mathcal  L}$, $\zeta = (\xi,\tau),
\zeta' = (\xi',\tau')\in T_z{\mathcal L} = T_xS \times \R$.

In the present section we are going to take advantage of Theorem
\ref{nuovo} to prove that each two points of $\elle$ are
geodesically connected in $\elle_n$,
for $n$ large enough. To avoid misunderstandings, the objects
associated to each spacetime
$\elle_n$ will be denoted by a
subindex $n$. So, the functional $f$ in \eqref{fuu} associated to
$\elle_n$ translates into
\be\label{statn} f_n(z)\ =\
\frac{1}{2}\int_{0}^{1}\langle\dot{z},\dot{z}\rangle_n \de s=\
\frac12\ \|\dot{x}\|_2^2\ + \int_0^1 \langle\delta(x),\dot
x\rangle \, \dot t \,\de s\ -\ \frac{1}{2n} \|\dot{t}\|_2^2. \ee
Analogously, the functional $\J$ in (\ref{Jn'}) becomes
\be\label{Jn}
\begin{split}
\J_n(x)\ =\ &\frac 12\|\dot{x}\|_2^2
+\frac n2\left[\int_0^1\<\delta(x),\dot x\>^2 \, \de s
-\left(\int_0^1 \<\delta(x),\dot x\> \, \de s\right)^2\right]\\
&-\Delta_t \left(\frac{\Delta_t}{2n} - \int_0^1\<\delta(x),\dot x\> \, \de s\right).
\end{split}
\ee
Furthermore, the geodesic equations \eqref{equa}, particularized to $\elle_n$ in
\eqref{stati}, translate into
\begin{equation}\label{metrican}
\left\{ \begin{array}{ll}
      {
\displaystyle
D_s\dot x - \dot t\,F(x)[\dot x] + \ddot t\,\delta(x) = 0} \\
{\displaystyle \frac {\de}{\de s}\left(\frac 1n\dot t -
\<\delta(x),\dot x\>\right) = 0.}
      \end{array}
      \right.
      \end{equation}

With these ingredients, now we can establish the announced result.

\begin{proposition}\label{p1} Let $(\elle,\langle\cdot,\cdot\rangle_L)$ be
a spacetime as in Theorem \ref{tm}. Given two points
$p=(x_p,t_p)$, $q=(x_q,t_q) \in \elle$ with $\Delta_t=t_q-t_p \ge
0$, there exists $n_0\in \N$ such that $p$ and $q$ are connected
by a geodesic $\gamma_n=(x_n,t_n)$ in $(\elle_n,
\langle\cdot,\cdot\rangle_n)$ for every $n\geq n_0$.
\end{proposition}

\begin{proof}
From Theorem \ref{nuovo} applied to each $\elle_n$, it suffices to
prove the existence of some $n_0\in\N$ such that
\be\label{uff}
J_n^-(q)\cap (S\times\{t_p\})\quad\hbox{ is compact in
$S\times\{t_p\}$ for all $n\geq n_0$.} \ee
Arguing by contradiction, assume that condition \eqref{uff} is false for
infinitely many $(\elle_m, \langle\cdot,\cdot\rangle_m)$. Then, by
the Hopf--Rinow Theorem, since $(S,\langle\cdot,\cdot\rangle)$ is
complete and $J_m^-(q)\cap (S\times\{t_p\})$ is closed
(Proposition \ref{p} {\sl (i)}), this last set cannot be bounded.
Hence, for each $m$, there exists an unbounded sequence of points
$(y^m_k)_k$ in $J_m^-(q)\cap (S\times\{t_p\})$. Then, by using a
Cantor's diagonal type argument applied to the family of these
sequences, for each $m$ there exists $k_m\in \N$ such that,
denoting $y_m=y^m_{k_m}$ with $y_m\in J_m^{-}(q)\cap (S\times
\{t_p\})$, the sequence $(y_m)_m$ is still unbounded. Let
$(\gamma_m)_m$ be a sequence of past inextendible
$\langle\cdot,\cdot\rangle_m$--causal curves departing from $q$
and passing through $y_m$ (recall Footnote $1$). Taking any $n_0
\in \N$, if $m\geq n_0$ then $\gamma_m$ is not only causal for
$\langle\cdot,\cdot\rangle_m$, but also for
$\langle\cdot,\cdot\rangle_{n_0}$ (by the metric expression
\eqref{stati}). From \cite[Proposition 3.31]{bee} applied to the
sequence of curves $(\gamma_m)_m$ in $({\mathcal
L},\<\cdot,\cdot\>_L)$, we obtain an inextendible limit curve
$\gamma=(x,t)$ departing from $q$, which is
$\langle\cdot,\cdot\rangle_n$--causal for all $n$, and thus,
$\langle\cdot,\cdot\rangle_L$--causal. Since $(\gamma_m)_m$
intersects $S\times \{t_p\}$ in an unbounded sequence of points,
the limit curve $\gamma$ cannot intersect $S\times\{t_p\}$, in
contradiction with the Cauchy character of the hypersurface
$S\times\{t_p\}$ in $({\mathcal L},\<\cdot,\cdot\>_L)$.
\end{proof}

\begin{remark}\label{minimi} {\em
We recall  that a $C^1$ functional $J\colon \Omega\to\R$, defined on
a Hilbert manifold $\Omega$, satisfies the {\em Palais--Smale condition} if
each sequence $(x_n)_n\subset \Omega$, such that $(J(x_n))_n$ is bounded
and $\de J(x_n)\to 0$ admits a converging subsequence.\\
The spatial components $x_n$ of the connecting geodesics $\gamma_n=(x_n,t_n)$
provided by Proposition \ref{p1} are minimum of the functionals $\J_n$ in \eqref{Jn}:
indeed, the $c$--precompactness  of $\Omega_K(p,q)$ for $\J_n$
for $n\geq n_0$ (cf. Theorem \ref{nuovo}), implies that the functionals $\J_n$ are bounded from below, satisfy the
Palais--Smale condition and have complete sublevels, so that they attain their infimum
(see \cite[Propositions 4.3 and 5.5, Theorem 5.3]{gp} and also \cite[Theorem 3.3]{bcc}).}
\end{remark}

\section{Proof of Theorem \ref{tm}}\label{s3}

Let $(\elle,\langle\cdot,\cdot\rangle_L)$ be a spacetime as in
Theorem \ref{tm}. In particular, by Proposition \ref{condelta} $\elle = S \times \R$ and
$\langle\cdot,\cdot\rangle_L$ is as in (\ref{metrica1}), with
metric coefficients given by Remark \ref{remark}. Consider two
points $p=(x_p,t_p)$, $q=(x_q,t_q) \in \elle$ with
$\Delta_t=t_q-t_p \ge 0$ and assume the existence of a $C^1$ curve
$\varphi=(y,t): I \to \elle$ connecting them such that
$\langle\dot{\varphi},K(\varphi)\rangle_L=\langle\delta(y),\dot{y}\rangle$
has constant sign or is identically equal to $0$.

Let $\left(\gamma_n=(x_n,t_n)\right)_{n\geq n_0}$ be the sequence
of curves connecting $p$ to $q$, each $\gamma_n$ geodesic in $\elle_n$,
as stated in Proposition \ref{p1}. Then, the following technical results hold:

\begin{lemma}\label{bounded} The sequence $\left(\|\dot{x}_n\|_2\right)_{n\geq n_0}$ is
bounded.
\end{lemma}

\begin{proof} Arguing by contradiction, assume that $(\|\dot{x}_n\|_2)_{n\ge
n_0}$ is not bounded. Taking any $\bar n \ge n_0$, the three
claims in the proof of Theorem \ref{nuovo} imply that
$\left({\mathcal J}_{\bar n}(x_n)\right)_{n\ge n_0}$ is not upper
bounded either. By the expression of the functionals in \eqref{Jn}
and the Cauchy--Schwarz inequality, it follows that
\[
{\mathcal J}_{n}(x_n)\geq {\mathcal J}_{\bar n}(x_n) \quad\hbox{
for all } n\geq\bar n.
\]
Whence, also $\left({\mathcal J}_{n}(x_n)\right)_{n\ge n_0}$ is
not bounded from above.\\
Next, assume that $\langle\delta(y),\dot{y}\rangle\not\equiv 0$ on
$I$. Then, the reparametrized curve $\tilde {y}(s)=y(r(s))$, with
\[
{r}(s)\ =\ \int_0^s\frac{1}{\langle\delta(y(r)),\dot
y(r)\rangle}\de r,\] satisfies
\[
\langle\delta(\tilde{y}(s)),\dot{\tilde{y}}(s)\rangle=\langle\delta(y(r)),\dot y(r)\rangle\,
\dot{r}(s) =1.
\]
In particular,
\[
\int_0^1\langle\delta(\tilde y),\dot{\tilde{y}}\rangle^2\,\de s -
\left(\int_0^1\langle\delta(\tilde y),\dot{\tilde{y}}\rangle\,\de
s\right)^2=0,
\]
and this equality holds also when
$\langle\delta(y),\dot{y}\rangle\equiv 0$ on $I$.
So, at any case we deduce
\begin{eqnarray*}
{\mathcal J}_n(\tilde{y})\ &=&\ \frac{1}{2}\|\dot{\tilde{y}}\|_2^2
- \Delta_t \left(\frac{\Delta_t}{2n} -
\int_0^1\langle\delta(\tilde y),\dot{\tilde{y}}\rangle\, \de s\right) \\
&\le&\ \frac{1}{2}\|\dot{\tilde{y}}\|_2^2 + \Delta_t
\int_0^1\langle\delta(\tilde y),\dot{\tilde{y}}\rangle\, \de
s\qquad \hbox{for all $n \in \N$.}
\end{eqnarray*}
Therefore, ${\mathcal J}_n(\tilde{y})$ admits an upper bound
independent of $n$, and thus
\[
{\mathcal J}_n(\tilde{y})\ <\ {\mathcal J}_n(x_n)\qquad\hbox{for
infinitely many $n$,}
\]
in contradiction with the minimum character of $x_n$, as stated in
Remark \ref{minimi}.
\end{proof}

\begin{lemma}\label{boundedt} The sequence $\left(\|\dot{t}_n\|_2\right)_{n\geq n_0}$ is
bounded.
\end{lemma}

\begin{proof}\footnote{Along this proof, for any integer $j\ge 1$ the constant $c_j$ will always denote a strictly
positive real number which does not depend on $s \in I$ and $n \ge n_0$.}
Taking the scalar product of the first equation in \eqref{metrican} applied to $\gamma_n=(x_n,t_n)$,
$n\geq n_0$, by the vector field $\delta(x_n)$, we get
\[
\langle D_{s}\dot{x}_n,\delta(x_n)\rangle -\dot{t}_n\langle F(x_n)[\dot{x}_n],\delta(x_n)\rangle
+\ddot{t}_n\langle\delta(x_n),\delta(x_n)\rangle \equiv 0 \quad\hbox{ on } I.
\]
So, $\tau_n=\dot{t}_n$ satisfies the first order linear ODE
\begin{equation}\label{e0}
\dot{\tau}_n=a_n(s)\, \tau_n + b_n(s) \quad\hbox{ on } I,
\end{equation}
where
\begin{equation}\label{e9}
a_n(s)=\frac{\langle F(x_n(s))[\dot{x}_n(s)],\delta(x_n(s))\rangle}{\langle\delta(x_n(s)),\delta(x_n(s))\rangle},\qquad
b_n(s)=-\frac{\langle
D_{s}\dot{x}_n(s),\delta(x_n(s))\rangle}{\langle\delta(x_n(s)),\delta(x_n(s))\rangle}
\end{equation}
($\delta$ is non-vanishing, recall Proposition \ref{condelta}).
Since
\begin{equation}\label{reg}
\int_0^1\dot{t}_n \, \de s=t_q-t_p=\Delta_t\quad\hbox{for all
$n\geq n_0$,}
\end{equation}
necessarily
\be\label{f} \dot{t}_n(s_n)=\Delta_t\quad\hbox{for some $s_n\in I$.}
\ee
So, $\dot{t}_{n}(s)$ is the unique solution
to (\ref{e0}) which satisfies condition (\ref{f}), i.e.
\begin{equation}\label{e6}
\dot{t}_n(s)=\tau_{n}(s)=e^{A_n(s)}\left(g_n(s) +\Delta_t\right),
\end{equation}
where $A_n(s)$ is the primitive of $a_n(s)$ satisfying $A_n(s_n)=0$ and, for simplicity, we have put
\be\label{gienne}
g_n(s) \ =\ \int_{s_n}^{s} b_n(r) e^{-A_n(r)} \de r.
\ee
Now, in order to prove the boundedness of $(\|\dot{t}_n\|_2)_{n\geq n_0}$,
firstly, we claim that
\be\label{e20}
c_1\leq e^{A_n(s)}\leq c_2 \quad\hbox{ on $I$, for all $n\geq n_0$.}
\ee
In fact, by applying inequality \eqref{pi} to $x_n$, Lemma \ref{bounded} implies that the sequence
\be\label{e21}
\hbox{ $\left(\|x_n\|_\infty\right)_{n\geq n_0}$ is bounded,}
 \ee
thus
\be\label{suD}
c_3\le \langle\delta(x_n(s)),\delta(x_n(s))\rangle \le c_4
\quad\hbox{ on $I$, for all $n\geq n_0$.}
\ee
Then, by the Cauchy--Schwarz inequality, \eqref{e9}, \eqref{e21} and \eqref{suD} we obtain
\be\label{suA}
|a_n(s)|\leq c_5 |\dot x_n(s)| \quad \hbox{ on $I$,}
\ee
with $|\dot x_n(s)|^2 = \langle\dot x_n(s),\dot x_n(s)\rangle$.
Hence, Lemma \ref{bounded} implies
\[
|A_n(s)|\leq c_6 \quad\hbox{ on $I$, for all $n\geq n_0$,}
\]
which  implies \eqref{e20}.\\
So, in order to conclude the proof, from \eqref{e6} and \eqref{e20}
it suffices to show that
\be\label{fboun}
\left(\|{g}_n\|_2\right)_{n\ge n_0} \hbox{ is bounded. }
\ee
To this aim, let us note that
\[
\langle D_{s}\dot{x}_n,\delta(x_n)\rangle=-\langle
\dot{x}_n,\frac{\de}{\de s}\delta(x_n)\rangle + \frac{\de }{\de
s}\langle \dot{x}_n,\delta(x_n)\rangle,
\]
thus by \eqref{e9} and \eqref{gienne}, integrating by parts we have
\be\label{e8}\begin{split}
g_n(s) \ =\ &\int_{s_n}^s\langle \dot{x}_n,\frac{\de}{\de r}\delta(x_n)\rangle\
\frac{e^{-A_n(r)}}{\langle\delta(x_n),\delta(x_n)\rangle}\ \de r\\
&\ -\
\int_{s_n}^s\frac{\de}{\de r}\big(\langle \dot{x}_n,\delta(x_n)\rangle\big)\
\frac{e^{-A_n(r)}}{\langle\delta(x_n),\delta(x_n)\rangle}\ \de r\\
=\ &\int_{s_n}^s\langle \dot{x}_n,\frac{\de}{\de r}\delta(x_n)\rangle\
\frac{e^{-A_n(r)}}{\langle\delta(x_n),\delta(x_n)\rangle}\ \de r\\
&\ -\
\frac{e^{-A_n(s)}\langle \dot{x}_n(s),\delta(x_n(s))\rangle}{\langle\delta(x_n(s)),\delta(x_n(s))\rangle}
+ \frac{e^{-A_n(s_n)}\langle \dot{x}_n(s_n),\delta(x_n(s_n))\rangle}{\langle\delta(x_n(s_n)),\delta(x_n(s_n))\rangle}\\
&\ +\ \int_{s_n}^s \langle \dot{x}_n,\delta(x_n)\rangle\
\frac{\de}{\de r}\left(\frac{e^{-A_n(r)}}{\langle\delta(x_n),\delta(x_n)\rangle}\right)\ \de r.
\end{split}
\ee 
The smoothness of $\delta$, \eqref{e20}--\eqref{suA}, the
Cauchy--Schwarz inequality, direct computations and Lemma
\ref{bounded} imply that for all $n\geq n_0$ the following bounds hold:
\begin{equation}\label{e10}
\left|\int_{s_n}^s\langle \dot{x}_n,\frac{\de}{\de
r}\delta(x_n)\rangle\
\frac{e^{-A_n(r)}}{\langle\delta(x_n),\delta(x_n)\rangle}\ \de
r\right|\ \leq c_7 \|\dot{x}_n\|_2^2 \le c_8,
\end{equation}
\begin{equation}\label{e11}
\left|\ \frac{\langle
\dot{x}_n(s),\delta(x_n(s))\rangle}{\langle\delta(x_n(s)),\delta(x_n(s))\rangle}\right|
\ \leq\ c_{9}|\dot{x}_n(s)| \quad\hbox{ on } I,
\end{equation}
\begin{equation}\label{e12}
\begin{split}
&\left| \int_{s_n}^s \langle \dot{x}_n,\delta(x_n)\rangle\
\frac{\de}{\de r}\left(\frac{e^{-A_n(r)}}{\langle\delta(x_n),\delta(x_n)\rangle}\right)\ \de r\right|\\
&\qquad \le \ \int_0^1 |\langle \dot{x}_n,\delta(x_n)\rangle| \ \frac{|a_n(r)| e^{-A_n(r)}}{\langle\delta(x_n),\delta(x_n)\rangle}\ \de r \\
&\qquad\quad +\
2 \int_0^1 |\langle \dot{x}_n,\delta(x_n)\rangle| \ e^{-A_n(r)}\
\frac{\big|\langle \delta(x_n),\frac{\de}{\de r}\delta(x_n)\rangle\big| }{\langle\delta(x_n),\delta(x_n)\rangle^2}\ \de r \\
&\qquad \ \le\ c_{10} \|\dot{x}_n\|_2^2 \le c_{11}.
\end{split}
\end{equation}
Moreover, we claim that 
\begin{equation}\label{e66}
\left|\ \frac{\langle
\dot{x}_n(s_n),\delta(x_n(s_n))\rangle}{\langle\delta(x_n(s_n)),\delta(x_n(s_n))\rangle}\right|
\ \leq\ c_{12}\|\dot{x}_n\|_2\ \leq\ c_{13} \quad\hbox{ on } I,
\end{equation}
In fact, from the second equality in (\ref{metrican}) 
we have 
\[
\frac 1n\dot{t}_n - \<\delta(x_n),\dot{x}_n\> \equiv k_n\quad\hbox{ on } I;
\]
thus, from one hand \eqref{f} implies 
\[
k_n\ =\ \frac{1}{n}\dot{t}_n(s_n)
- \<\delta(x_n(s_n)),\dot{x}_n(s_n)\>\ =\ 
\frac{\Delta_t}{n}-\langle\delta(x_n(s_n)),\dot{x}_n(s_n)\rangle,
\]
while, from the other hand, \eqref{reg} gives
\[
k_n\ = \ \int_{0}^{1}\left(\frac 1n\dot{t}_n(s) -
\<\delta(x_n(s)),\dot{x}_n(s)\>\right)\de s
\ =\ \frac{\Delta_t}{n}-\int_{0}^{1}\langle\delta(x_n(s)),
\dot{x}_n(s)\rangle \de s.
\]
Whence,
\[
 \langle\delta(x_n(s_n)),\dot{x}_n(s_n)\rangle\ =\ \int_{0}^{1}\langle\delta(x_n(s)),\dot{x}_n(s)\rangle \de s
\]
and \eqref{e66} follows from \eqref{suD} and, again, Lemma
\ref{bounded}.\\
At last, by using \eqref{e10}--\eqref{e66} in \eqref{e8}, we
have that
\[
|g_n(s)|\ \leq\ c_{14} |\dot x_n(s)| + c_{15} \quad\hbox{ on $I$,
for all $n\geq n_0$;}
\]
whence, Lemma \ref{bounded} implies \eqref{fboun}.
\end{proof}

\begin{lemma}\label{strong} There exists $\gamma=(x,t)\in\Omega(x_p,x_q;S)\times W(t_p,t_q)$
such that, up to subsequences, $\left(\gamma_n\right)_{n\geq n_0}$
strongly converges to $\gamma$ on $\Omega(x_p,x_q;S)\times W(t_p,t_q)$.
\end{lemma}

\begin{proof}
From \eqref{pi} and Lemmas \ref{bounded}, \ref{boundedt}, the
sequences $\left(\|x_n\|\right)_{n\geq n_0}$ and
$\left(\|t_n\|\right)_{n\geq n_0}$ are bounded, thus there exists
$\gamma=(x,t)\in H^1(I,\R^N)\times H^1(I,\R)$ such that, up
to subsequences, \be\label{xdeb} x_n \rightharpoonup x \quad\hbox{
weakly in $H^1(I,\R^N)$ (and also uniformly in $I$)} \ee and
\[
t_n \rightharpoonup t \quad\hbox{weakly in }  H^1(I,\R).
\]
Furthermore, as $S$ is complete, by \eqref{xdeb} it follows that
$x\in\Omega(x_p,x_q;S)$ and there exist two sequences
$(\xi_n)_{n\geq n_0}$, $(\nu_n)_{n\geq n_0}$ in $H^1(I,\R^N)$ such
that
\be\label{weak}
\begin{split}
&\xi_n\in T_{x_n}\Omega (x_p,x_q;S), \quad x_n-x=\xi_n + \nu_n \quad\hbox{ for all } n\geq n_0,\\
&\xi_n\rightharpoonup 0 \quad\hbox{weakly} \quad\hbox{ and } \quad
\nu_n\rightarrow 0 \quad\hbox{strongly in } H^1(I,\R^N)
\end{split}
\ee
(cf. \cite[Lemma 2.1]{bf}).
Taking any $n\geq n_0$, by Proposition \ref{p1} and
\eqref{statn} we have $\de f_n(\gamma_n)[\zeta]=0$ for all $\zeta\in
T_{\gamma_n}\Omega_n(p,q)$, thus in particular \be\label{crit}
\begin{split}
&\int_0^1\langle\dot x_n,\dot \xi_n\rangle \, \de s
+\int_0^1\langle\delta'(x_n)\xi_n,\dot x_n\rangle\,\dot t_n\, \de s
+\int_0^1 \langle\delta(x_n),\dot \xi_n\rangle\,\dot t_n \, \de s\\
&\qquad - \int_0^1 \langle\delta(x_n),\dot x_n\>\,\dot \tau_n \, \de s + \int_0^1\frac 1n\dot t_n\dot\tau_n\, \de s\ =\ 0
\end{split}
\ee
for $\zeta = (\xi_n,-\tau_n)\in T_{\gamma_n}\Omega_n(p,q)$ with
$\tau_n = t_n - t \in H^1_0(I,\R)$. On the other hand, by Lemmas
\ref{bounded}, \ref{boundedt}  and \eqref{weak}, it results
\[
\int_0^1\langle\delta'(x_n)\xi_n,\dot x_n\rangle\,\dot t_n\, \de s\ =\ o(1),
\]
where $o(1)$ denotes an infinitesimal sequence. Whence,
\eqref{crit} implies
\[
\begin{split}
&\int_0^1\langle\dot x_n,\dot \xi_n\rangle \, \de s +
\int_0^1\frac 1n\dot t_n\dot\tau_n\, \de s \\
&\qquad = - \int_0^1 \langle\delta(x_n),\dot \xi_n\rangle\,\dot t_n \, \de s
+ \int_0^1 \langle\delta(x_n),\dot x_n\>\,\dot \tau_n \, \de s + o(1).
\end{split}
\]
Reasoning as in \cite[Theorem 3.3]{gm}, the strong convergence of
$(\gamma_n)_{n\geq n_0}$ to $\gamma$, up to a subsequence, is
deduced.
\end{proof}

\begin{proof}[Proof of Theorem \ref{tm}]
The implication $(i) \Longrightarrow (ii)$ is a direct consequence
of \eqref{vinc}.\\
For the implication $(ii) \Longrightarrow (i)$, let
$\left(\gamma_n=(x_n,t_n)\right)_{n\geq n_0}$ be the sequence of curves
connecting $p$ to $q$, with each $\gamma_n$ geodesic in $\elle_n$, provided by
Proposition \ref{p1}. From Lemma \ref{strong} there exists a curve
$\gamma=(x,t)\in\Omega(x_p,x_q;S)\times W(t_p,t_q)$ such that, up
to subsequences, \be\label{strongly} x_n \to x \;\hbox{ strongly
in $\Omega(x_p,x_q;S)$}\quad\hbox{ and }\quad t_n \to t \;\hbox{
strongly in $W(t_p,t_q)$.} \ee It suffices to prove that $\gamma$
satisfies equations \eqref{equa} with $\beta\equiv 0$, i.e.,
\begin{equation}\label{metricaluce}
\left\{ \begin{array}{l}
 D_s\dot x - \dot t\,F(x)[\dot x] + \ddot t\,\delta(x) = 0, \\
{\displaystyle \frac {\de}{\de s}\left( \<\delta(x),\dot x\>\right) = 0.}
      \end{array}
      \right.
      \end{equation}
To this aim, let us remark that if $n \ge n_0$, by Theorem \ref{t1} applied to $f_n$
in \eqref{statn}, we have \be\label{da}
\de f_n(\gamma_n)[\zeta]=0\quad\hbox{ for all } \zeta\in
T_{\gamma_n}\Omega(p,q). \ee
Then in particular, taking any $\tau \in
H^1_0(I,\R)$ and $\zeta = (0,\tau)$ in \eqref{da}, it follows that
\[
\int_0^1\langle\delta(x_n),\dot x_n\rangle\,\dot\tau\,\de s - \frac 1n\int_0^1\dot t_n\dot\tau\,\de s\ =\ 0;
\]
hence, passing to the limit, by \eqref{strongly} we get
\[
\int_0^1\langle\delta(x),\dot x\rangle\dot\tau\,\de s=0.
\]
Thus, for the arbitrariness of $\tau\in H^1_0(I,\R)$
 the second equality in \eqref{metricaluce} holds.\\
On the other hand, taking any $\eta\in T_x\Omega (x_p,x_q;S)$, by
\eqref{strongly} and \cite[Lemma 2.2]{bf} there exists a sequence
$(\eta_n)_{n\geq n_0}$, with $\eta_n\in T_{x_n}\Omega (x_p,x_q;S)$,
converging weakly to $\eta$. Then, choosing $\zeta=(\eta_n,0)$ in \eqref{da}
for $n \ge n_0$, by passing to the limit
and taking into account \eqref{strongly}, we obtain
\[
\int_0^1\langle\dot x,\dot \eta\rangle\,\de s + \int_0^1\langle\delta'(x)\eta,\dot x\rangle\dot t\,\de s +
\int_0^1\langle\delta(x),\dot \eta\rangle\dot t\,\de s\ =\ 0 .
\]
Therefore, integrating by parts and for the arbitrariness of $\eta\in T_x\Omega (x_p,x_q;S)$,
we deduce that $\gamma=(x,t)$ is smooth
and verifies the first equation in \eqref{metricaluce}. Hence, the proof is complete.
\end{proof}

The proof of Theorem \ref{tm} requires global hyperbolicity only
in two points: for ensuring the decomposition \eqref{metrica1} and
for proving the following property:
\smallskip

\noindent {\em (*) $\quad$ Any past inextendible causal curve
departing from $q=(x_q,t_q)$, $t_q\geq t_p$, must intersect
$S\times \{t_p\}$.}

\smallskip

\noindent Therefore, if we are dealing with a spacetime which
already splits globally as in \eqref{metrica1}, the global
hyperbolicity assumption can be replaced by property {\em (*)}. More
precisely, the same arguments performed in the proof of Theorem \ref{tm}
allow us to state the following generalization:

\begin{theorem}\label{nuovo'}
Let $({\mathcal L},\<\cdot,\cdot\>_{L})$ be a spacetime with
$\elle=S\times\R$ and $\langle\cdot,\cdot\rangle_{L}$ as in
(\ref{metrica1}). Assume that $(S,\langle\cdot,\cdot\rangle)$ is a
complete Riemannian manifold. Given two points $p=(x_p,t_p)$,
$q=(x_q,t_q)$, with $\Delta_t= t_q- t_p\geq 0$, satisfying
property (*), the following statements are equivalent:
\begin{itemize}
\item[$(i)$] $p$ and $q$ are geodesically connected in $\elle$;
\item[$(ii)$] $p$ and $q$ can be connected by a $C^1$ curve
$\varphi=(y,t)$ on $\elle$ such that $\langle \delta(y),
\dot{y}\rangle$ has constant sign or is identically equal to $0$.
\end{itemize}
\end{theorem}

\section{Accuracy of the hypotheses of Theorem \ref{tm}.}\label{s8}

\noindent (a) {\em Counterexample if the lightlike Killing vector
field is not complete.}

\smallskip

\noindent Consider the spacetime obtained by removing from the
Minkowski $2$--space ${\mathbb L}^{2}$ the region $\{(x,t): x\geq
0,\, t\geq 0\}$. This spacetime admits the hyperplane $t\equiv -1$
as a complete Cauchy hypersurface, and $K= \partial_x +
\partial_t$ as a non--complete lightlike Killing vector field.
However, the points $p=(1,-1)$, $q=(-1,1)$, which can be connected
with a $C^1$ curve $\varphi$ with
$\langle\dot{\varphi},K(\varphi)\rangle_L$ having constant
negative sign, cannot be connected by a geodesic.

\bigskip

\noindent (b) {\em Counterexample if the Cauchy hypersurface is
not complete.}

\smallskip

\noindent Consider $\elle=S\times\R$, $S=\R^2\setminus\{(x_1,0):
-1\leq x_1\leq 1\}$ equipped with the Lorentzian metric
\[
\langle\zeta,\zeta'\rangle_L=\langle\xi,\xi'\rangle_0+\langle\delta(x),\xi\rangle_0\,
\tau' + \langle\delta(x),\xi'\rangle_0\, \tau,
\]
for all $\zeta = (\xi,\tau), \zeta' = (\xi',\tau')\in \R^3$, where
$\langle\cdot,\cdot\rangle_0$ is the canonical scalar product on
$S\subset\R^2$ and $\delta: x=(x_1,x_2)\in S \mapsto
\lambda(x)(1,0)\in \R^2$, with $\lambda$ a positive smooth
function on $S$ such that $\langle\cdot,\cdot\rangle_0/\lambda^2$
is complete on $S$. Note that $K=\partial_t$ is a complete
lightlike Killing vector field and $S\times\{t\}$ is a
non--complete Cauchy hypersurface for every $t\in\R$ (apply
\cite[Proposition 3.1]{s} with $F_n(x)\equiv 2 \lambda(x)$ for all
$n$). However, this manifold is not geodesically connected. In
fact, consider two points $p=(x_p,0)$, $q=(x_q,0)$ with
$x_p=(0,-1)$, $x_q=(0,1)$. By the second equation in
\eqref{metricaluce}, any geodesic $\gamma=(x,t)$ joining $p$
to $q$ must satisfy
\[
\frac{\de}{\de s}\langle\delta(x),\dot{x}\rangle_0=0,
\]
but the sign of $\langle\delta(x),\dot{x}\rangle_0$ must
change for any curve $x=x(s)$ departing from $x_p$ and arriving to
$x_q$. Hence, there are no geodesics connecting $p$ to $q$.

\bigskip

\noindent (c) {\em The existence of a complete lightlike Killing
vector field and a complete Cauchy hypersurface do not imply
geodesic conectedness.}

\smallskip

\noindent Consider $\elle=\R^3\times\R$ equipped with the
Lorentzian metric
\[
\langle\zeta,\zeta'\rangle_L=\langle\xi,\xi'\rangle_0+\langle\delta(x),\xi\rangle_0\,
\tau' + \langle\delta(x),\xi'\rangle_0\, \tau,
\]
for all $\zeta = (\xi,\tau), \zeta' = (\xi',\tau')\in \R^4$, where
$\langle\cdot,\cdot\rangle_0$ is the canonical scalar product on
$\R^3$ and $\delta: x=(x_1,x_2,x_3)\in \R^3 \mapsto \delta(x_1)\in
\R^3$ satisfies
\[
\delta(x_1)=\left\{\begin{array}{ll} (-\cos^3 x_1,0,0) & \hbox{if
$x_1<\pi$} \\ (1,0,0) & \hbox{if $x_1\geq\pi$.}\end{array}\right.
\]
In this spacetime $\partial_t$ is a complete lightlike Killing
vector field and $\R^{3}\times\{t\}$ is a complete Cauchy
hypersurface for every $t\in\R$ (apply \cite[Proposition 3.1]{s}
with $F_n\equiv 2$ for all $n$). However, this spacetime is not
geodesically connected. In fact, for any curve $x=x(s)$
departing from a point in $\R^3$ with $x_1=0$ and arriving to a
point in the region $x_1>\pi$, the sign of
$\langle\delta(x),\dot{x}\rangle_0$ must change. Hence, reasoning as in the
previous item, there is no geodesic which connects the points $p=(x_p,0)$ and
$q=(x_q,0)$, where, for example, it is $x_p=(0,0,0)$ and $x_q=(3\pi/2,0,0)$.

\section{Some Applications}\label{s7}

\subsection{Avez--Seifert result.} A first consequence of Theorem \ref{tm} is that it provides the
classical Avez--Seifert result (cf., e.g., \cite[Theorem
3.18]{bee}) in our ambient:
\begin{proposition}\label{pp} Let $(\elle,\<\cdot,\cdot\>_L)$ be a globally hyperbolic spacetime
endowed with a complete lightlike Killing vector field $K$ and a
complete Cauchy hypersurface $S$. Then, two points of $\elle$ can
be connected by a causal geodesic if and only if they are causally
related.
\end{proposition}

\begin{proof}
 We will focus on the implication to the left, as the
converse is trivial. So, assume that two points $p,q\in \elle$ are
causally related. Then, they are connectable by a $C^1$ causal
curve $\varphi=(y,\tau)$, which, up to a reparameterization,
satisfies that $\langle\dot\varphi, K(\varphi)\rangle_L$ is
constant. Thus, from Theorem \ref{tm} the points $p$ and $q$ are
connectable by a geodesic
$\gamma=(x,t)$.\\
In order to prove that $\gamma=(x,t)$ is causal, it suffices to
show that $f(\gamma)\leq 0$. To this aim, recall that
$\gamma=(x,t)$ can be approached by a sequence of geodesics
$\gamma_n=(x_n,t_n)$, $n\geq n_0$, of
$(\elle_n,\<\cdot,\cdot\>_n)$, where each $x_n$ is a minimum of the
functional $\J_n$ (recall  Remark \ref{minimi} and Lemma
\ref{strong}). So, from one hand,
$\gamma_n \to \gamma$ strongly in $\Omega(x_p,x_q;S)\times W(t_p,t_q)$
 (and also uniformly in $I$) and the boundedness of
$\left(\|\dot{x}_n\|_2\right)_{n\geq n_0}$ and
$\left(\|\dot{t}_n\|_2\right)_{n\geq n_0}$ imply
\[
\J_n(x_n)=f_n(\gamma_n)\rightarrow f(\gamma)\quad\hbox{as
$n\rightarrow\infty$}
\]
(cf. also \cite[Theorem 3.3]{gm}). On the other hand,
\[
\J_n(x_n)\leq \J_n(y)=f_n(\varphi)\rightarrow f(\varphi)\leq 0
\quad\hbox{as $n\rightarrow\infty$.}
\]
In conclusion, $f(\gamma)\leq 0$ and, thus, $\gamma$ is causal.
\end{proof}

\subsection{Generalized plane waves.} Theorem \ref{tm}
becomes also useful for studying the geodesic connectedness of a family of
Lorentzian manifolds which generalizes the gravitational waves,
the so--called generalized plane waves (see \cite{MTW}).

\begin{definition}\label{pfwave}
{\rm A Lorentzian manifold $(\elle,\<\cdot,\cdot\>_L)$ is called
\emph{generalized plane wave}, briefly \emph{GPW}, if there exists
a (connected) finite dimensional Riemannian manifold
$(\M,\<\cdot,\cdot\>)$ such that $\elle = \M \times \R^{2}$ and
\[
\<\cdot,\cdot\>_L\ =\ \<\cdot,\cdot\> + 2 du dv+ \h (x,u) du^{2},
\]
where $x\in \M$, the variables $(u,v)$ are the natural coordinates
of $\R^{2}$ and the smooth function $\h: \M\times \R\rightarrow
\R$ is not identically zero.}
\end{definition}
A GPW becomes a gravitational wave if $\M = \R^2$ is equipped with
the classical Euclidean metric and $\h(x,u)\ =\ g_1(u) (x_1^2 -
x_2^2) + 2 g_2(u) x_1x_2$, $x = (x_1,x_2)\in \R^2$, for some
smooth real functions $g_1$ and $g_2$ such that $g_1^2 + g_2^2
\not\equiv 0$ (for more details, cf., e.g., \cite{bee}).

The geodesic connectedness and the global hyperbolicity of GPWs
have been investigated in \cite{cfs1, FS}. In particular, if the
Riemannian manifold $(\M,\<\cdot,\cdot\>)$ is complete with
respect to its canonical distance $d(\cdot,\cdot)$ and $\h$
behaves subquadratically at spatial infinity, i.e., there exist
$\bar{x} \in \M$ and (positive) continuous functions $R_1(u)$,
$R_2(u)$, $p(u)$, with $p(u) < 2$, such that
\[
- \h(x, u) \le R_1(u) d^{p(u)}(x,\bar{x}) + R_2(u) \quad \hbox{for
all $(x,u) \in \M\times\R$,}
\]
then the spacetime is not only geodesically connected (cf.
\cite[Corollary 4.5]{cfs1}) but also globally hyperbolic (cf.
\cite[Theorem 4.1]{FS}). This suggests an intrinsic connection
between these two properties, as the following simple consequence
of our approach confirms:

\begin{theorem}\label{c0}
Any globally hyperbolic GPW with a complete Cauchy hypersurface is
geodesically connected.
\end{theorem}

\begin{proof}
Let $(\elle,\<\cdot,\cdot\>_L)$ be a GPW. Clearly, $K=\partial_v$
is a complete lightlike Killing vector field on $\elle$. Take any
$p=(x_p,u_p,v_p), q=(x_q,u_q,v_q) \in \elle$, any curve $x=x(s)$
in $\M$ connecting $x_p$ to $x_q$, and denote $\Delta_u=u_q-u_p$
and $\Delta_v=v_q-v_p$. The curve $\varphi(s)=(x(s),\Delta_u\,
s,\Delta_v\, s)$ connects $p$ to $q$, and the scalar product
\[
\langle \dot{\varphi}, K(\varphi)\rangle_L = \dot u=\Delta_u
\]
has constant sign or is equal to 0. Therefore, the existence of a
geodesic connecting $p$ to $q$ follows from Theorem \ref{tm}.
\end{proof}

\begin{remark} {\em
To the authors it is not clear if any globally hyperbolic GPW
$(\elle,\<\cdot,\cdot\>_L)$ with $(\M,\<\cdot,\cdot\>)$ complete,
necessarily admits some complete Cauchy hypersurface\footnote{This
question is in connection with the following more general problem,
which goes beyond the scope of the present article: find general
conditions on a globally hyperbolic spacetime which ensure that it
admits some complete Cauchy hypersurface.}. If this was true, in
the hypotheses of Theorem \ref{c0} this last condition could be
replaced by the completeness of $(\M,\<\cdot,\cdot\>)$.}
\end{remark}


\end{document}